\documentclass[11pt]{amsart}
\usepackage{fancyhdr}
\usepackage[english]{babel}
\usepackage{amssymb}
\usepackage{amsmath}
\usepackage{amsthm}
\usepackage{bbold}
\usepackage{mathrsfs}
\usepackage{enumerate}
\usepackage{xcolor}
\usepackage{ifpdf}
\usepackage[all]{xy}
\usepackage{tikz-cd}
\usepackage[initials]{amsrefs}
\usepackage{color}

\numberwithin{equation}{section}

\newcommand{\bbN}{{\mathbf N}}
\newcommand{\bbR}{{\mathbf R}}

\newcommand{\bbC}{{\mathbf C}}
\newcommand{\bbE}{\mathbb{E}}

\newcommand{\calA}{\mathcal{A}}

\newcommand{\calH}{\mathcal{H}}

\newcommand{\calX}{\mathcal{X}}
\newcommand{\calY}{\mathcal{Y}}

\newcommand{\iv}{^{-1}}

\newcommand{\supp}{\operatorname{supp}}

\newcommand{\Prob}{\operatorname{Prob}}

\newcommand{\ra}{\rightarrow}

\newcommand{\dd }{\,{\rm d}}

\newtheorem{mthm}{Theorem}

\newtheorem{theorem}{Theorem}[section]
\newtheorem{lemma}[theorem]{Lemma}

\newtheorem{claim}{Claim}

\newtheorem{corollary}[theorem]{Corollary}

\newtheorem{proposition}[theorem]{Proposition}
\newtheorem{prop}[theorem]{Proposition}

\newtheorem{assumption}[theorem]{Assumption}

\begin{document}

\title[Lattice of Boundaries and the Entropy Spectrum]{On the Lattice of Boundaries and the Entropy Spectrum of Hyperbolic Groups}
\author{Samuel Dodds}
\address{University of Illinois at Chicago}
\email{sdodds3@uic.edu}

\date{\today}

\maketitle

\begin{abstract}
Let $\Gamma$ be a non-elementary hyperbolic group and $\mu$ be a probability on $\Gamma$. We study the $\mu$-proximal, stationary actions, also known as boundary actions, of $\Gamma$. In particular, we are interested in the spectrum of Furstenberg entropies of $(\Gamma,\mu)$-boundaries, and the lattice-theoretic and topological structure of the set $\mathcal{BL}(\Gamma,\mu)$ of boundaries. We prove that all hyperbolic groups have infinitely many distinct boundaries, which attain an infinite set of distinct entropies. Additionally, for simple random walks on non-abelian free groups $F_d$, we establish that there are infinitely many boundaries whose entropy is greater than $\frac{1}{2}-\epsilon$ times the entropy of Poisson boundary, when the rank $d$ is large. General results of independent interest about the order-theoretic and continuity properties of Furstenberg entropy for countable groups are attained along the way. This includes the result that under mild assumptions, the spectrum of boundary entropies $\calH_{\text{bound}}(\Gamma,\mu)$ is closed.
\end{abstract}

%%%%%%%%%%%%%%%%%%%%%%%%%%%%%%%%%%%%%%%%%%%%%%%%%%%%%%%%%%%%%%%%%%%%%%%%%%%%%%%%%%%%%%%%%%%%%%%%%%%

\section{Introduction}

Let $\Gamma$ be a group, and $\mu \in \Prob(\Gamma)$ be a probability measure on $\Gamma$. The notion of a {\it $(\Gamma,\mu)$-boundary} was put forward by Furstenberg (\cite{Furst1}, \cite{Furst2}) to generalize harmonic analysis and the study of random walks to non-abelian groups. Briefly, a measure space $(B,\calA,\nu)$ is a $(\Gamma,\mu)$-boundary if it is {\it $(\Gamma,\mu)$-stationary} ($\mu * \nu = \int_\Gamma g\nu \dd\mu(g)$) and {\it $(\Gamma,\mu)$-proximal}, which means that after choosing a sequence $g_i\in \Gamma$ independently according to $\mu$, then $g_1 \ldots g_n\nu$ almost surely concentrates into a Dirac mass.  

There are typically many stationary boundaries associated to a single measured group $(\Gamma, \mu)$. There is however, a common structure: for a fixed $(\Gamma,\mu)$, the collection of all stationary boundaries $\mathcal{BL}(\Gamma,\mu)$ is an (order-theoretic) lattice. The order is given by $\Gamma$-equivariant factor maps, if $(B_1,\nu_1) \ra (B_2,\nu_2)$ then we say $(B_1,\nu_1) \geq (B_2,\nu_2)$. For any pair of boundaries the join and meet exist and are denoted by $(B_1\vee B_2, \nu_1 \vee \nu_2)$ and $(B_1\wedge B_2, \nu_1 \wedge \nu_2)$, respectively. The maximal element of $\mathcal{BL}(\Gamma,\mu)$ is the {\it Poisson Boundary} $(\partial_P(\Gamma,\mu),\nu_P)$ and the minimal element is always the trivial boundary $(\ast,\delta_\ast)$. Thus we may think of all boundaries as arising as factors of the Poisson boundary.

Entropy is a key tool in the study of stationary spaces. The particular variety of entropy which is most adapted to this situation is that of {\it Furstenberg Entropy}: if $(X,\lambda)$ is a stationary $(\Gamma,\mu)$-probability space, then set

\begin{equation} \label{entropy}
h_\mu(X,\lambda) = \int_\Gamma \int_X -\log \left ( \frac{\dd g\iv \lambda }{\dd \lambda}(x) \right) \dd \lambda(x) 
\end{equation}

Furstenberg entropy is decreasing under factors, and strictly decreasing under non-trivial factors between boundaries. The entropy of the trivial boundary is 0, and the entropy of the Poisson boundary $\partial_P(\Gamma,\mu)$ is finite under the assumption that $H(\mu) = \sum_{g \in \Gamma} -\log(\mu(g)) \mu(g) < \infty$. We will take this as granted for the remainder of the paper.

\begin{assumption}
For each measured group $(\Gamma,\nu)$, the Shannon entropy is finite: $H(\mu) \leq \infty$
\end{assumption}

Kaimanovich-Vershik (\cite{KV}) show that one can compute the entropy of the Poisson boundary via $h_\mu(\partial_P(\Gamma,\mu, \nu_P) = h_{RW}(\mu)$ where 

\[
h_{RW}(\mu) = \lim_{n\ra\infty} \frac{H(\mu^n)}{n}
\]

\noindent is Avez's \cite{Avez} random walk entropy. Entropy can be viewed an order preserving map from the lattice of boundaries $\mathcal{BL}(\Gamma,\mu)$ to $\bbR$. We will refer to the image of this map as the {\it boundary entropy spectrum}, and denote it by 
\[
\calH_\text{bound}(\Gamma,\mu)=\{h_\mu(B,\nu) ~|~ (B,\nu) \text{ is a } (\Gamma,\mu)\text{-boundary}\}
\] 
Similarly, if we extend the domain of the Furstenberg entropy functional to all $(\Gamma,\mu)$-stationary spaces, we call the image the {\it stationary entropy spectrum}, denoted by 
\[
\calH_\text{stat}(\Gamma,\mu)=\{h_\mu(X,\lambda) | (X,\lambda) \text{ is a stationary} (\Gamma,\mu)\text{-space}\}
\]  
Understanding the behaviour of the $\calH_\text{bound}(\Gamma,\mu)$ is quite difficult in general, but one interesting piece of information is Kazhdan's Property (T). Let $\mu$ be finitely supported, symmetric and generating, then according to Nevo (\cite{Nevo}) and Bowen-Hartman-Tamuz \cite{BHT} $\Gamma$ has property (T) if and only if 0 is isolated in $\calH_\text{stat}(\Gamma,\mu)$. For $\calH_\text{bound}(\Gamma,\mu)$ only the ``only if" statement is known. For fixed $(\Gamma,mu)$ a bit more is understood. If $\Gamma$ is a lattice in a semisimple Lie group, or the fundamental group of a compact negatively curved Riemannian manifold, then probabilities $\mu$ can be constructed that ``discretize Brownian motion" in some sense (see \cite{Furst1}), so that the hitting measures at the boundaries of the associated symmetric spaces or universal covers, respectively, are precisely the Lebesgue measure. In the higher rank case, one can then apply Margulis' factor theorem \cite{Margulis_factor} to see that, $\mathcal{BL}(\Gamma,\mu)$ is a finite lattice, and thus $\calH_\text{bound}(\Gamma,\mu)$ is finite as well.

On the other hand, Bowen \cite{Bowen} has shown that for $\Gamma =F_d$ a free group and $\mu$ uniform on a symmetric free generating set, then $\calH_\text{stat}(\Gamma,\mu)$ is an interval, in fact, one can take the spaces which attain this range of entropy to be ergodic. Hartman-Tamuz \cite{HT} show that $0$ is not isolated in $\mathcal{BL}(\Gamma,\mu)$ when $\Gamma$ is virtually free, and $\mu$ is any generating measure with finite first moment. Tamuz-Zheng show that $\calH_\text{bound}(\Gamma,\mu)$ contains an infinite set with no isolated points \cite{TZh}. See also \cite{HY}. Here, we will show that there are boundaries whose entropy is at least a fixed proportion of the maximal entropy.

\begin{mthm}\label{middle H thm}
Let $\Gamma = F_d$ be a non-abelian free group, and $\mu$ be a uniform probability on a symmetric free generating set. Then there is a family of boundaries $(B_n,\nu_n)$, so that for any $\epsilon>0$ there is an $n$ such that
\[
h_{RW} > h_\mu(B_n,\nu_n) > \frac{d-2}{2d-2}h_{RW}(\mu)-\epsilon
\]
\end{mthm}
 
Clearly, this estimate becomes stronger as $d \ra \infty$. In forthcoming work with Alex Furman, we establish that there is an interval of the form $[0,\epsilon]$ in $\calH_\text{bound}(F_d,\mu)$. With this one can strengthen the above result to state that there is a interval of entropies that are attained near the middle of the possible range of entropies as $d \ra \infty$.

In the more general setting of hyperbolic groups, we show that an countably branching, countably deep rooted tree embedds into $\mathcal{BL}(\Gamma,\mu)$. That is to say, consider the set of finite sequences of natural numbers $\bbN^{<\infty} = \cup_{n=0}^\infty \bbN^n$ with a partial order so that a sequence $\alpha = (a_1,\ldots, a_n)$ is greater than a sequence $\beta=(b_1,\ldots,b_m)$ iff $n<m$ and for $i\leq n$, $a_i=b_i$. For $\alpha,\beta \in \bbN^{<\infty}$, let $\alpha \vee \beta$ be the longest sequence so that the beginning of $\alpha$ and $\beta$ agree with $\alpha \vee \beta$. One can think of this as a countably branching, countable depth rooted tree, and $\alpha \vee \beta$ is the location that the branch containing $\alpha$ and $\beta$ first join. Now we can state our theorem:

\begin{mthm}\label{hyperbolic boundaries}
Given a non-elementary hyperbolic group $\Gamma$ and a probability $\mu$ on $\Gamma$ with finite logarithmic moment $\sum_{g \in \Gamma} \log |g| \mu(g) < \infty$, then there is a family of distinct, non-trivial boundaries $\{(B_\alpha,\nu_\alpha) | \alpha \in \bbN^{<\infty} \}$ so that $B_\alpha \vee B_\beta  = B_{\alpha \wedge \beta}$. Moreover, there is constant $c>0$ so that $h_\mu(B_\alpha, \nu_\alpha) > c$. 
\end{mthm}

\begin{corollary}
For every non-elementary hyperbolic group $\Gamma$, and probability measure $\mu$ with finite logarithmic moment, both $\mathcal{BL}(\Gamma,\mu)$ and $\calH_\text{bound}(\Gamma,\mu)$ are infinite.
\end{corollary}

In a different direction, there are several topologies that one can place on $\mathcal{BL}(\Gamma,\mu)$ that are compatible with the lattice structure. Given that all stationary boundaries are quotients of the Poisson boundary, we may conflate boundaries with $\Gamma$-invariant sub $\sigma$-algebras of the $\sigma$-algebra $\calA_P$ on the Poisson boundary. This allows one to use any one of the various topologies that have been defined on the set of sub-$\sigma$-algebras. Given a standard probability space $(X,\calX,\lambda)$, we say that a sequence $\calA_n$ of sub-$\sigma$-algebras of $\calX$ converges in the {\it $L^p$-strong topology} to $\calA$ if the conditional expectation operators corresponding to $\calA_n$ converges to $\calA$ in the $L^p(X,\lambda)$ strong operator topology. In other words $\calA_n \ra_p \calA$ if and only if for all $f \in L^p(X,\lambda)$

\[ \lim_{n \ra \infty} \| \bbE(f | \calA_n) - \bbE(f | \calA) \|_p =0 \]

In Bj\"{o}rklund-Hartman-Oppelmayer \cite{BHO} consider $L^1$-strong convergence and Kudo limits, and in this setting prove that entropy is continuous, and that convergence to the Poisson or trivial boundary is implied by the convergence of entropy to the maximal and minimal value respectively. Recall that given a collection of boundaries, $\{(B_\sigma,\nu_\sigma) | \sigma \in \Sigma\} \subseteq \mathcal{BL}(\Gamma,\mu)$ the supremum $(B^+,\nu^+) = \sup_\sigma(B_\sigma,\nu_\sigma)$ of $\Sigma$ is the smallest boundary so which is greater than all the $(B_\sigma,\nu_\sigma)$, and the infimum $(B^-,\nu^-)$ is defined mutatis mutandis. We show that Furstenberg entropy is compatible for sufficiently monotone families of boundaries in $\mathcal{BL}(\Gamma,\mu)$. 

\begin{mthm}\label{order and entropy thm}
Given a collection $\{(B_\sigma,\nu_\sigma) ~|~ \sigma \in \Sigma\}$ of boundaries closed under joins and meets, then the supremum $(B^+,\nu^+)$ is the unique boundary so that
\begin{enumerate}
    \item for all $\sigma \in \Sigma$, $(B^+,\nu^+)\geq(B_\sigma,\nu_\sigma)$
    \item $h_\mu(B^+,\nu^+) = \sup_\sigma h_\mu(B_\sigma,\mu_\sigma)$
\end{enumerate}
Assuming that $\mu$ has finite support, then the infimum $(B^-,\nu^-)$ of the collection is the unique boundary so that
\begin{enumerate}
    \setcounter{enumi}{2}
    \item for all $\sigma \in \Sigma$, $(B^-,\nu^-)\leq(B_\sigma,\nu_\sigma)$
    \item $h_\mu(B^-,\nu^-) = \inf_\sigma h_\mu(B_\sigma,\mu_\sigma)$
\end{enumerate}
\end{mthm}

The proof is elementary in the sense that it does not may use of any additional information about $\Gamma$ or any structures on $\mathcal{BL}(\Gamma,\mu)$ besides the order. While this result is purely order theoretic, the next result is concerned with the topological properties of $\mathcal{BL}(\Gamma,\mu)$.

\begin{mthm}\label{entropy is continuous thm}
For a measured group $(\Gamma,\mu)$ so that $H(\mu) < \infty$, the boundary lattice $\mathcal{BL}(\Gamma,\mu)$ is a compact topological lattice when equipped with the $L^2$-strong topology, and the map $(B,\nu) \mapsto h_\mu(B,\nu)$ sending a boundary to its Furstenberg entropy is a continuous, hence closed map.
\end{mthm} 

% Leaving the world of operator topologies, Sayag-Shalom and Sayag (\cite{SaSh} and \cite{Sayag}) study the space of boundedly quasi-invariant actions of $\Gamma$: $BQI(\Gamma,F)$ is the space of quasi-invariant $\Gamma$ probability spaces $(X,m)$ for which Radon-Nikodym derivatives are bounded by a function $F:\Gamma \ra \bbR$, i.e. $\left|\frac{\dd gm}{\dd m}(x) \right|\leq F(g)$ $m$-a.e.. They show that $BQI(\Gamma,F)$ is closed under ultralimits, and that Furstenberg entropy is continuous under ultralimits; that is to say that the entropy of an ultralimit of spaces in $BQI(\Gamma, F)$ is the ultralimit of the entropies of those spaces. We prove the following 

% \begin{mthm}\label{ultralimit containment}
% For any finitely supported generating measure $\mu$ on $\Gamma$, then $\mathcal{BL}(\Gamma,\mu) \subset BQI(\Gamma,M)$ for some function $M$, and is closed under ultralimits. 
% \end{mthm}

One easily obtains the corollary:

\begin{corollary}\label{spectrum closed}
The boundary entropy spectrum $\calH_\text{bound}(\Gamma,\mu)$ is compact, and hence closed.
\end{corollary} 

In light of this and \cite{TZh}, we can say that for the uniform random walk on the free group, $\calH_\text{bound}(F_d,\mu)$ contains a continuum sized, perfect set.

%%%%%%%%%%%%%%%%%%%%%%%%%%%%%%%%%%%%%%%%%%%%%%%%%%%%%%%%%%%%%%%%%%%%%%%%%%%%%%%%%%%%%%%%%%%%%%%%%%%%

\section{Background}

\subsection{Stationary Spaces}

Let $\Gamma$ be a countable discrete group, and $\mu$ a probability on $\Gamma$. Suppose that $( X,\calX,\lambda)$ is a Lebesgue probability space with a measurable $\Gamma$ action. The convolution of $\mu$ and $\lambda$ is $\mu*\lambda = \int_\Gamma g \lambda \dd \mu(g)$. Briefly, we say that $(X, \calX, \lambda)$ is {\it $(\Gamma,\mu)$-stationary} (or simply {\it stationary}) if $\nu$ is invariant under convolution by $\mu$, i.e. $\mu * \nu  = \nu$. Define the convolution powers of $\mu$ via $\mu^n = \mu * \cdots * \mu \in \Prob(\Gamma)$. We can assume that up to measurable isomorphism that $X$ is compact, and the action of $\Gamma$ is continuous. The next lemma due to Furstenberg \cite{Furst1} is fundamental to the study of stationary spaces.

\begin{lemma}\label{FurstenbergLemma}
Let $(X,\nu)$ be a compact stationary space, then for $\mu^\bbN$-a.e. $\gamma = (g_i) \in \Gamma^\bbN$ the weak* limit
\[ 
\lim_{n \ra \infty} g_1\ldots g_n \nu = \nu_\gamma 
\]
exists, and the map from $\Gamma^\bbN$ to $\Prob(X)$ via $\gamma \mapsto \nu_\gamma$ is measurable.
\end{lemma}

A {\it $(\Gamma,\mu)$-boundary} is a stationary space $(B,\nu)$ which is also {\it $\mu$-proximal}: for $\mu^\bbN$-almost every walk $\omega$ on $\Gamma$, the limit $\nu_\gamma = \delta_{\beta_B(\omega)}$ is a Dirac mass. 

According to lemma \ref{FurstenbergLemma} any measurable $\Gamma$-equivariant factor $(X,\lambda)$ of a boundary $(B,\nu) \ra (X,\lambda)$ is also a boundary, both stationarity and proximality are inherited by factors. Thus every boundary can be realized as a factor of, or identified with a complete $\sigma$-algebra of, the unique maximal boundary, called the Poisson boundary. For any of the equivalent constructions of the Poisson boundary see \cite{Furman} or \cite{KV}.

We say that a measure $\mu$ is {\it generating} if $\Gamma = \bigcup_{n=0}^\infty\supp(\mu^n)$. If $\mu$ is generating then the $\Gamma$ action on any stationary space $(X,\lambda)$ is measure class preserving. 

\begin{lemma}\label{finitelysupported}
For a countable group $\Gamma$, and a generating, finitely supported probability $\mu$. Then for any stationary $(\Gamma,\mu)$ space $(X,\nu)$, There is a positive function $M:\Gamma \ra (1,\infty)$ so that the Radon-Nikodym derivative $\frac{\dd g \nu}{\nu}(x)$ exists and for $\nu$-a.e. $x$ we have
\[
 \frac{1}{M(g)} \leq \frac{\dd g \nu}{\dd \nu}(x) \leq M(g)
\]
\end{lemma}
\begin{proof}
Since $\mu$ is generating, for any $g \in \Gamma$ there is an $n$ so that $\mu^n(g)>0$. As $\nu$ is $\mu$-stationary, we have 
\[
1= \sum_{h\in\Gamma} \frac{\dd h \nu}{\dd \nu}(x) \mu^n(h) \geq \frac{\dd g \nu}{\dd \nu}(x) \mu^n(g).
\]
Hence $\frac{\dd g \nu}{\dd \nu}(x) \leq \frac{1}{\mu^n(g)}$. In a similar way one can show that for some $k$, $0<\mu^k(g^{-1}) \leq \frac{\dd g\nu}{\dd \nu}(x)$. So we can take $M(g) = \max \left \{\frac{1}{\mu^k(g^{-1})},\frac{1}{\mu^n(g)} \right\} $. 
\end{proof}

For a more complete discussion of stationary actions of groups, see the article of Furstenberg-Glasner \cite{FurstGlas}.

\subsection{Conditional Measures and Expectations}
Recall that for any factor map of probability spaces $f:(X, \nu) \rightarrow (Y, \eta)$,  $\{ \nu_y\} \subseteq \Prob(Y)$ is a {\it system of conditional measures} for $\nu$ over $\eta$ if $\nu  = \int_Y \nu_y ~d\eta $ and $\supp(\nu_y) \subseteq f^{-1}(y)$. Following the work of Rohklin \cite{Roh}, such a system exists for any measurable map of standard probability spaces, and moreover such a system is unique in the sense that if $\nu^1_y$ and $\nu^2_y$ are systems of conditional measures for $\nu$ over $\eta$ then $\nu^1_y=\nu^2_y$.

\begin{lemma}\label{fiberderivative}
Let $f:(X,\nu) \rightarrow (Y, \eta)$ be a $\Gamma$-equivariant factor map of standard probability spaces with a measure class preserving $\Gamma$ actions. For convenience if $x \in X$, set $y = f(x)$. Then for all $g \in \Gamma$ and $\nu$-a.e. $x \in X$, then $[g\nu_{g^{-1}y}] = [\nu_y]$ and, 
    \begin{align*}
        \frac{\dd g\nu}{\dd\nu}(x) = \frac{\dd g\nu_{g^{-1}y}}{\dd\nu_y}(x) \cdot \frac{\dd g\eta}{\dd\eta}(y)
    \end{align*}
\end{lemma}

\begin{proof}
First we see that for $\eta$ a.e. $y$ the measures $\nu_y$ and $g\nu_{g\iv y}$ are in the same measure class. This follows easily from the essential uniqueness of conditional measures, the fact that $[g\nu] = [\nu]$, and the formula 

\[g \nu = \int_Y g \nu_{g\iv y}~\dd \eta(x)\]
 
which we will establish presently. Take $\psi$ to be an arbitrary bounded measurable function on $X$, then
    \begin{align*}
        \int_X \psi(x) ~\dd g\nu(x) 
            &= \int_Y \int_X \psi(gx) ~\dd \nu_y(x) \dd\eta(y) \\
            &= \int_Y \int_X \psi(x) ~\dd g\nu_y(x) \dd\eta(y)\\
            &= \int_Y \int_X \psi(x) ~\dd g\nu_{g^{-1}y}(x) \dd g\eta(y) 
    \end{align*}
    
From this we get the desired property
    \begin{align*}
        \int_X \psi(x) ~\dd g\nu(x) &= \int_Y \int_X \psi(x) ~\dd g\nu_{g^{-1}y}(x) \dd g\eta(y) \\
            &= \int_Y \int_X \psi(x) \frac{\dd g\nu_{g^{-1}y}}{\dd\nu_y}(x) \cdot \frac{\dd g\eta}{\dd\eta}(y) ~\dd \nu_{y}(x) \dd\eta(y)\\
            &= \int_X \psi(x) \frac{\dd g\nu_{g^{-1}y}}{\dd\nu_y}(x) \cdot \frac{\dd g\eta}{\dd\eta}(y) ~\dd\nu(x)
    \end{align*}
\end{proof}

Given a measure space $(X,\calX,\lambda)$, a $\sigma$-subalgebra $\calA$ of $\calX$, and a measurable function $f \in L^1(X,\calX,\lambda)$, the {\it condtional expectation $\bbE(f|\calA)$} is defined as the unique $\calA$ measurable function so that for any other $\calA$ measurable function $h$ then 
\[ \int_X\bbE(f|\calA)\cdot h \dd\lambda = \int_X f\cdot h \dd \lambda\]
In the case that $\calA$ is the pullback $\sigma$-algebra in $X$ of a factor map $\pi:(X,\calX,\lambda) \ra (Y,\calY, \rho)$, we can form a system of conditional measures $\lambda_y$. Then $\bbE(f | \calA)(x) = \int_X f(s) \dd\lambda_{\pi(x)}(s)$ is the condition expectation operator associated to $\calA$. There is a bijective correspondence between $\sigma$-subalgebras and conditional expectation operators.

\subsection{Hyperbolic Groups}
By the work of Bowditch \cite{Bowditch}, we may think of a hyperbolic group $\Gamma$ as a group which admits a certain kind of proximal action on compact metrizable space $\partial\Gamma$. That is to say, $\Gamma$ is {\it hyperbolic} if there is a minimal $\Gamma$ action $\Gamma \curvearrowright \partial\Gamma$ so that the induced action on distinct triples $\Gamma \curvearrowright \partial\Gamma^{(3)}$ is properly discontinuous and cocompact. The space $\partial \Gamma$ is then called the {\it Gromov Boundary} of $\Gamma$. The assumptions of minimality and cocompactness insure that $\partial\Gamma$ is the unique $\Gamma$-space satisfying the above conditions up to $\Gamma$-equivariant homomorphism. The {\it elliptic radical} $E(\Gamma)$ of $\Gamma$ is the kernel of the action of $\Gamma$ on its Gromov boundary $\partial \Gamma$. Every element of $\gamma$ is either in $E(\Gamma)$ or it fixes a finite set in $\partial \Gamma$.

% \subsection{Ultralimits}
% We will give a basic outline of the construction of ultralimits of measure spaces. For the full technical details of the construction of the ultralimit of measure spaces we refer to \cite{Sayag}, \cite{SaSh} or.A {\it non-principle ultrafilter} on the natural numbers $\bbN$, is a finitely additive, 0-1 valued, probabiliy measure defined on the full power set $\calP(\bbN)$, and so that no finite set has measure 1. Fix a non-principle ultrafilter $\omega$ on $\bbN$ for the remainder of the paper. 

% Given a bounded sequence $x_n \in \bbR$ we say that $x$ is the {\it ultralimit} denoted $\omega\lim_n x_n=x$, if for every $\epsilon>0$, $\omega(\{n ~|~|x_n-x|<\epsilon \}) =1$. The definition of the ultrafilter ensures that every bounded sequence has an unique ultralimit.

% Given a sequence of measure spaces $(X_n,\nu_n)$, we can construct the {\it ultraproduct} $X = \prod_n X_n/~$ where $(x_n) ~(x'_n)$ iff $\omega(\{n| x_n = x'_n\})$. Given a sequence of measurable sets $S_n \subset X_n$ consider the {\it internal set} $S = \prod_n S_n/~ \subset X$. We can attempt to define the measure $\nu$ of $S$ via 
% \[ \nu(S)= \omega\lim_n \nu_n(S_n) \]

%%%%%%%%%%%%%%%%%%%%%%%%%%%%%%%%%%%%%%%%%%%%%%%%%%%%%%%%%%%%%%%%%%%%%%%%%%%%%%%%%%%%

\section{Proof of Theorem \ref{order and entropy thm}}

We will break theorem \ref{order and entropy thm} into two propositions, proved separately. With lemma \ref{fiberderivative} in hand we can prove the following proposition about the relation between the supremum of sequences of boundaries, and the supremum of sequences of boundary entropies

\begin{proposition} \label{increasinglimits}
Let $\Gamma$ be a countable group and $\mu$ any probability on $\Gamma$. Take $\{(B_i,\nu_i)~|~ i\in \bbN \}$ be a increasing sequence of $(\Gamma,\mu)$-boundaries: i.e. if $i<j$ then there is a factor map $\pi_{ij}:(B_j,\nu_j) \ra (B_i ,\nu_i)$. Let $\sup_i (B_i,\nu_i)$ be the smallest boundary so that $(B,\nu) \ra (B_i,\nu_i)$ implies $(B,\nu) \ra \sup_i (B_i,\nu_i)$. Then
\begin{enumerate}
    \item $h_\mu(\sup_i (B_i,\nu_i)) = \lim_i h_\mu(B_i,\nu_i) = \sup_i h_\mu(B_i,\nu_i)$ 
    \item Suppose that there is a $(B,\nu)$ so that $h_\mu(B,\nu) = \lim_i h_\mu(B_i,\nu_i) = \sup_i h_\mu(B_i,\nu_i)$ and $B$ factors over all the $B_i$ ($\tau_i:B \ra B_i)$, then 
    $(B,\nu)=\sup (B_i,\nu_i)$
\end{enumerate}
\end{proposition}

\begin{proof}
\underline{\it Of part 1} : For convenience we will adopt the notation $\sup_i (B_i,\nu_i) = (B^+,\nu^+)$. Consider the functions on $\Gamma \times B^+$ defined by
\[ 
C_i(g,b) = -\log \left(\frac{\dd g \nu_i}{\dd \nu_i}(\tau_i(b)) \right) 
\]
The finite support of $\mu$ implies that $|C_i(g,b)| \leq \inf\{\frac{1}{\mu^n(g)}~|~\mu^n(g) >0\}$ is a uniform bound on $C(g,b)$ in the $b$ coordinate for every $i$,  If $(\rho_{ij})_{x} $ the conditional measure of $\nu_j$ over $\nu_i$ at $x \in B_i$, then by lemma \ref{fiberderivative} we have that 
\[
C_j(g,b) = C_i(g,b) - \log \left(\frac{\dd g (\rho_{ij})_{g^{-1}\tau_i(b)} }{\dd (\rho_{ij})_{\tau_i(b)}}(\tau_j(b)) \right).
\]
Denote the $\sigma$-algebra on $B^+$ by $\calA^+$. If $\calA_i$ the complete sub-$\sigma$-algebra of $\calA^+$ associated to $B_i$ then the set $\{\calA_i ~|~ i\in \bbN\}$ is a increasing filtration, and by Jensen's inequality
\begin{align*}
    \bbE(C_j|\calA_j) &= C_i + \bbE(- \log \left(\frac{\dd g (\rho_{ij})_{g^{-1}\tau_i(b)} }{\dd (\rho_{ij})_{\tau_i(b)}}(\tau_j(b)) \right)|\calA_j)\\
    &\geq C_i  -\log \left( \bbE( \left(\frac{\dd g (\rho_{ij})_{g^{-1}\tau_i(b)} }{\dd (\rho_{ij})_{\tau_i(b)}}(\tau_j(b)) \right)|\calA_j) \right) \\
    &= C_i -\log(1) = C_i
\end{align*}

Thus as functions of $b \in B^+$, we may think of $C_i(g,b)$ a submartingale with uniformly bounded integrals $\int_{B^+}C_i($, which by the martingale convergence theorem (\cite{Doob}), must converge $\nu$-a.e. and in $L^1(B^+,\nu^+)$. Let us call the limit $C(g,b)$, and let $\psi(g,b) = \exp{-C(g,b)}$. We now prove that $\psi(g,b) = \frac{\dd g \nu^+}{\dd \nu^+}(b)$, which by the $L^1$ convergence of $C_i$, will prove the part 1. Notice that $\frac{\dd g \nu_i}{\dd \nu_i}(\tau_i(b))$ converges pointwise a.e. to $\psi(g,b)$. If $f$ is a bounded measurable function on $B^+$, then let $f_i = \bbE(f | \calA_i)$. Since $f_i$ is $\calA_i$ measurable there is a function $\bar{f_i}:B_i\ra \bbC$ so that $f_i = \bar{f_i}\circ \tau_i$. By dominated convergence with dominating function $\|f\|_\infty \psi $

\begin{align*}
    \int_{B^+} f(b) \psi(g,b) \dd \nu^+(b) &= \int_{B^+} \lim_i \left( f_i(b) \frac{\dd g \nu_i}{\dd \nu_i}(\tau_i(b)) \right)\dd \nu^+(b) \\
        &= \lim_i \int_{B_i} \bar{f_i}(b) \frac{\dd g \nu_i}{\dd \nu_i}(\tau_i(b)) \dd \nu_i(b) \\
        &= \lim_i \int_{B^+} f_i(b)  \dd g\nu_i(b) \\
        &= \int_{B^+} \lim_i f_i(b)  \dd g\nu^+(b) \\
        &= \int_{B^+} f(b) \frac{\dd g \nu^+}{\dd \nu^+}(b)  \dd \nu^+(b) \\
\end{align*}
Thus  $h_\mu(\sup_i (B_i,\nu_i)) = \lim_i h_\mu(B_i,\nu_i) = \sup_i h_\mu(B_i,\nu_i)$.

\underline{\it Of part 2} : By the definition of the supremum of a chain of boundaries given above $(B,\nu) \ra \sup (B_i,\nu_i)$. By part 1, if $h_\mu(B,\nu) = \sup h_\mu(B_i,\nu_i)$, then $h_\mu(B,\nu) = h_\mu(B^+,\nu^+)$. This implies $(B,\nu) = (B^+,\nu^+)$.
\end{proof}

For decreasing chains of boundaries, we will make certain additional assumptions about the distribution of the measure $\mu$,

\begin{proposition} \label{decreasing limits}
Let $\Gamma$ be a countable group, and $\mu$ be a finitely supported generating probability on $\Gamma$. Take $\{(B_i,\nu_i)~|~ i\in \bbN \}$ be a decreasing sequence of $(\Gamma,\mu)$-boundaries: i.e. if $i<j$ then there is a factor map $\pi_{ij}:(B_i,\nu_i) \ra (B_j ,\nu_j)$. Let $\inf_i (B_i,\nu_i)$ be the smallest boundary so that $(B_i,\nu_i) \ra (B,\nu)$ implies $\inf_i (B_i,\nu_i) \ra (B,\nu) $. Then
\begin{enumerate}
    \item $h_\mu(\inf_i (B_i,\nu_i)) = \lim_i h_\mu(B_i,\nu_i) = \inf_i h_\mu(B_i,\nu_i)$ 
    \item Suppose that there is a $(B,\nu)$ so that $h_\mu(B,\nu) = \inf_i h_\mu(B_i,\nu_i)$ and all of the $B_i$ factor over $B$ ($\tau_i:B_i \ra B)$, then 
    $(B,\nu)=\inf (B_i,\nu_i)$
\end{enumerate}
\end{proposition}

\begin{proof}
\underline{\it Part 1} : For notational convenience denote $\inf (B_i,\nu_i)$ by $(B^-,\nu^-)$. Let $\calA_i $ be the $\sigma$-algebra on $B_1$ which is the pullback of the $\sigma$-algebra on $B_i$ via the map $\pi_{1i} = \pi_{i}:B_1 \ra B_i$. Let $i<j$, if $\psi$ is an $\calA_j$ measurable function on $B_1$, then we can compute
\begin{align*}
    \int_{B_1}\psi \bbE\left(\frac{\dd g \nu_i}{\dd \nu_i} \circ \pi_i ~|~ \calA_j \right) \dd \nu_1 &= \int_{B_1}\psi \frac{\dd g \nu_i}{\dd \nu_i} \circ \pi_i  \dd \nu_1 \\
    &= \int_{B_1}\psi \frac{\dd g \nu_j}{\dd \nu_j} \circ \pi_j \dd \nu_1
\end{align*}
Thus for $i<j$, then $\bbE(\frac{\dd g \nu_i}{\dd \nu_i} \circ \pi_{i} ~|~\calA_j) = \frac{\dd g\nu_j}{\dd \nu_j} \circ \pi_j$. Now according to Doob (\cite{Doob} theorem 4.2) $\frac{\dd g\nu_i}{\dd \nu_i}(\pi_i(b))$ converges pointwise to $\bbE( \frac{g\nu_1}{\nu_1}~|~\bigcap_{i=1}^\infty\calA_i)$, and  $\bigcap_{i=1}^\infty\calA_i$ is the complete $\sigma$-subalgebra corresponding to $(B^-,nu^-)$. But this implies the pointwise convergence of $-\log \left( \frac{\dd g \nu_i}{\dd \nu_i} \right) \ra -\log \left( \frac{\dd g \nu^-}{\dd \nu^-} \right)$. By \ref{finitelysupported} and dominated convergence, we have that $h_\mu(B^-,\nu^-) = \lim_i h_\mu(B_i,\nu_i)$.

\underline{\it Part 2} : This follows readily from the equality of entropy estabished in part 1, and the definition of the infimum boundaries.
\end{proof}

\begin{proof}{\it (of theorem \ref{order and entropy thm})}
Given a collection of boundaries $\{(B_\sigma,\nu_\sigma)~|~\sigma \in \Sigma \}$ closed under finite joins and meets, let $H = \sup_\sigma h_\mu(B_\sigma, \nu_\sigma)$, and denote $(B^+,\nu^+) = \sup_\sigma (B_\sigma,\nu_\sigma)$. There must be a countable subset of indices $\sigma_i$ so that $\sup_{i\in \bbN} h_\mu(B_{\sigma_i},\nu_{\sigma_i}) = \lim_{i\ra \infty}h_\mu(B_{\sigma_i},\nu_{\sigma_i}) = H$. For simplicity call $B_{\sigma_i} = B_i$, consider the increasing chain of boundaries 
\[
(B'_i,\nu'_i) = \bigvee_{k = 1}^i (B_i,\nu_i) 
\]
The supremum of the entropies $h_\mu(B'_i,\nu'_i)$ along this chain is at least $H$, but on the other hand $(B'_i,\nu'_i)$ is part of the collection, as the collection is closed under joins. So the supremum of $h_\mu(B'_i,\nu'_i)$ is exactly $H$. Now applying the proposition \ref{increasinglimits} to this increasing chain, one prove the first half of the theorem.

The second half is indentical except that one uses meets instead of joins, and cites proposition \ref{decreasing limits}.
\end{proof}

\section{Proof of Theorems \ref{entropy is continuous thm}}

\begin{proof}{\it(of theorem \ref{entropy is continuous thm})} 
Let $B_i$ be a sequence of boundaries, and let $\calA_i$ be the corresponding complete $\sigma$-algebras of the Poisson boundary. Similarly, let $B$ be a boundary and $\calA$ be the corresponding $\sigma$-algebra. We wish to show that $L^2$ strong convergence implies the convergence of entropy. That is, if for all $f \in L^2(\partial_P\Gamma,\nu_P)$, $\|\bbE(f|\calA_i) -\bbE(f|\calA)\| \ra 0$ then $h_\mu(B_i, \nu_i) \ra h_\mu(B,\nu).$ 

Notice by Jensen's inequality, $\|\cdot\|_1 \le \|\cdot\|_2$, which implies that  $L^2(\partial_P\Gamma,\nu_P)$ is densely embedded in $L^1(\partial_P\Gamma,\nu_P)$. 
We now will show that $\calA_i$ converges $L^1$-strongly to $\calA.$ Take $h$ in $L^1(\partial_P\Gamma,\nu_P)$. Take $f \in L^2(\partial_P\Gamma,\nu_P)$ so that $\|h-f\|_1 < \epsilon.$ Now using the fact that conditional expectations have operator norm 1, 
\begin{align*}
    \|\bbE(h|\calA) - \bbE(h|\calA_i)\|_1 &\le \|\bbE(h - f |\calA_i)\|_1 + \| \bbE(f|\calA_i) - \bbE(f|\calA)\|_1 \\
    &\hspace{1cm} + \|\bbE(h-f|\calA)\|_1 \\
    &\le 2\|h-f\|_1 + \| \bbE(f|\calA_i) - \bbE(f|\calA)\|_1 \\
    &\le 2\|h-f\|_1 + \| \bbE(f|\calA_i) - \bbE(f|\calA)\|_2 \\
    &< 3\epsilon
\end{align*}

\noindent for sufficiently large $i$. Thus, convergence in the $L^2$ strong operator topology implies convergence in the $L^1$ strong operator topology for conditional expectation operators. 
Now, by \cite{BHO} $L^1$ strong convergence implies entropy convergence. Hence, Furstenberg entropy is continuous in the $L^2$ strong operator topology.

Now, \cite{BT} states that the space of $\sigma$-subalgbras of a Lebesgue measure space is compact in the $L^2$ strong topology.

\end{proof}

\begin{proof}{\it(of corollary \ref{spectrum closed} using theorem \ref{entropy is continuous thm})}
Since the Furstenberg entropy is continuous on $\mathcal{BL}(\Gamma,\mu)$, which is compact by the above, then its range is compact and thus closed. 
\end{proof}

It is worth noting that in \cite{BT} the authors prove that joins and meets are continuous in the $L^2$ strong topology. This creates the possibility for the following to raise or lower a ranges of entropies previously constructed.

% \begin{proof}{\it(of theorem \ref{ultralimit containment})}
% If a sequence of boundaries ($B_i, \nu_i$) has an ultralimit $(X, \lambda)$, then consider the factor maps from the Poisson boundary $\pi_i: (\partial_P \Gamma, \nu_P) \ra (B_i, \nu_i)$. By Lemma 4.16 of \cite{SaSh} parts 1 and 2, there is a map $\pi_\infty: (\partial_P \Gamma, \nu_P) \ra (X,\lambda)$. This implies that $(X,\lambda)$ is in fact a boundary. Thus, $\mathcal{BL}(\Gamma,\mu)$ is closed under ultralimits. 
% \end{proof}

% \begin{proof}{\it(of corollary \ref{spectrum closed} using theorem \ref{ultralimit containment})}
% In this case, we must assume that Let $(B_i,\nu_i)$ be a sequence of boundaries in $\mathcal{H}_\text{bound}(\Gamma,\mu)$, and suppose that $H = \lim_{i \ra \infty} h_\mu(B_i,\nu_i)$. Consider the ultralimit $(B_\infty, \nu_\infty)$ of the boundaries $(B_i, \nu_i)$, by Corollary 4.19 of \cite{SaSh} the entropy of $(B_\infty, \nu_\infty)$ is the ultralimit of the entropies of $(B_i, \nu_i)$. In the case that a limit exists, the ultralimit is equal to the limit; $\omega\text{-}\lim_{i \ra \infty} h_\mu(B_i,\nu_i) = H$, so the ultalimit of entropies is also $H$. Therefore, $h_\mu(B_\infty, \nu_\infty) = H$, and $\mathcal{H}_\text{bound}(\Gamma,\mu)$ is closed.
% \end{proof}

\section{Proof of Theorem \ref{middle H thm}}
This section concerns the range of boundary entropies for a simple random walk $\mu_s$ on a free group. A reasonable guess as to the structure of $\calH_\text{bound}(F_d,\mu_s)$ is that it should be the full interval $[0,h_{RW}(\mu)].$
It is well known that the asymptotic entropy of a finitely supported random walk is bounded above by the exponential rate of growth of the support of the walk at time $n$. The following lemma is a relative version of this fact.  Recall that the {\it  growth} of a group $\Gamma$ with respect to a word norm $\|\cdot\|$ is $v(\Gamma) = \lim_{n \ra \infty} \frac{1}{n} \log|B(n)|$, where $B(n)$ is the  ball of radius $n$ according to $\|\cdot\|$ centered at the identity. Given a quotient $\Gamma'$ of $\Gamma$, with kernel $N\triangleleft\Gamma$, then the {\it critical exponent of $N$} inside $\Gamma$ is $\delta(N) = \lim_{n \ra \infty} \frac{1}{n} \log|B(n)\cap N|$

\begin{lemma}\label{cogrowth bounds entropy}
Let $\Gamma' = \Gamma/N$ be a quotient of a group $\Gamma$, $\mu$ be a finitely supported, symmetric, generating probability measure on $\Gamma$, and $\mu'$ be the projection of $\mu$ to $\Gamma'$. Then the difference in the Furstenberg entropies of the Poisson boundaries of $(\Gamma,\mu)$ and $(\Gamma',\nu')$ can be estimated by 
\[
h_\mu(\partial_P (\Gamma,\mu),\nu_P) - h_{\mu'}(\partial_P (\Gamma',\mu'),\nu_P') \leq \delta(N)
\]
\end{lemma}

\begin{proof}
First note that for any coset $gN \in \Gamma'$ we may take a representative $g \Gamma$ of $gN$ so that $|g|_S$ is minimized, and if we do so, then $g^{-1}(gN \cap B(n) \subset N \cap B(n)$. Thus $|gN \cap B(n)| \leq |N \cap B(n)|$. 

We know from Kaimanovich-Vershik \cite{KV} the entropy of the Poisson boundary is equal to the Avez asymptotic entropy of the random walk on $\Gamma$. I.e.
\[
h_\mu(\partial_P(\Gamma,\mu),\nu_P) = h_{RW}(\mu) = \lim_{n \ra \infty} \frac{H(\mu^n)}{n}
\]
where  $H$ is the Shannon entropy $H(\mu^n) = \sum_{g \in \Gamma} -\log(\mu^n(g))\mu^n(g)$. Now we calculate via Jensen's inequality:
\begin{align*}
    h_\mu(\partial_P (\Gamma,\mu),\nu_P) &- h_{\mu'}(\partial_P (\Gamma',\mu'),\nu_P')  = h_{RW}(\mu)-h_{RW}(\mu')\\
    &= \lim_{n \ra \infty} \frac{1}{n} \Bigg(\sum_{g \in \Gamma} -\log(\mu^n(g))\mu^n(g) \\
    &\qquad - \sum_{gN \in \Gamma'} -\log(\mu'^n(gN))\mu'^n(gN) \Bigg)\\
    &= \lim_{n \ra \infty} \frac{1}{n} \Bigg(\sum_{gN \in \Gamma'}\sum_{x\in gN} -\log(\mu^n(x))\mu^n(x) \\
    &\qquad- \sum_{gN \in \Gamma'}\sum_{x \in gN} -\log(\mu^n(gN))\mu^n(x) \Bigg)\\
    &= \lim_{n \ra \infty} \frac{1}{n} \sum_{gN \in \Gamma'} \left( \sum_{x \in gN} \log\left(\frac{\mu^n(gN)}{{\mu}^n(x)}\right)\frac{\mu^n(x)}{\mu^n(gN)}\right)\mu^n(gN) \\
    &\leq \lim_{n \ra \infty} \frac{1}{n} \sum_{gN \in \Gamma'} \log|gN \cap \supp(\mu^n)| ~ \mu^n(gN) \\
    &\leq \lim_{n \ra \infty} \frac{1}{n} \log|N \cap B(n)| = \delta(N)
\end{align*}
\end{proof}

This estimate is limited in many groups due to the phenomena of {growth tightness}, the critical exponent of infinite normal subgroups is at least some fixed, positive proportion of $v(\Gamma)$. This is known to hold in free groups, and in fact in all non-elementary hyperbolic groups. 

\begin{proof}{\it(of theorem \ref{middle H thm})}
For the simple random walk $\mu$ on the free group $F_d$, the fundamental inequality of Guivarc'h attains equality. This means,
\[
h_{RW}(\mu) = d(\mu) \cdot v(F_d) 
\]
where $d(\mu)$ is the {\it drift} $d(\mu)= \lim_{n\ra \infty}\frac{1}{n} \sum_{g\in\Gamma} |g|\mu^n(g)$ of the random walk according to $\mu$ and the word metric $|\cdot |$. This allows to compute the random walk entropy exactly in terms of the rank:
\begin{equation}\label{H formula}
h_{RW}(\mu) = \frac{d-1}{d} \cdot \log(2d-1) 
\end{equation}
According to Jaerich-Matsuzaki \cite{JMgrowth} for all normal subgroups $N$ of $F_d$ we have $\delta(N) \geq \frac{1}{2}v(F_d)$, but there are sequences $N_i$  of normal subgroups so that $\delta(N_i) \ra \frac{1}{2}v(F_d)$ as $i\ra \infty$. Now set $\Gamma_i= \Gamma/N_i$, by \ref{cogrowth bounds entropy} for any $\epsilon>0$ there is a sufficiently large $k$ so that for $i>k$
\[
h_{RW}(\mu)- h_{\mu}(\partial_P \Gamma_i,(\nu_P)_i) \leq  \frac{1}{2}\log(2d-1)+\epsilon
\]
By \ref{H formula}, we obtain
\[
\frac{d-2}{2d-2}\log(2d-1)-\epsilon \leq h_{\mu}(\partial_P \Gamma_i,(\nu_P)_i) 
\]
This establishes the theorem.
\end{proof}

In work with Alex Furman \cite{Otherpaper}, we show that there is an explicit constant $\epsilon(d)$ so that if $\mu$ is the uniform probability on a symmetric free generating set in the free group $F_d$ of rank $d$, then there are essentialy free $(F_d,\mu)$-boundaries $(B_t,\nu_t)$ for each $t \in (0,\epsilon(d))$ so that $h_\mu(B_t,\nu_t) = t$. In fact all of the boundaries $(B_t,\nu_t)$ are realized by Lebesgue class measures on the sphere $S^2$. If follows from the density this construction, and the density of smooth functions in $L^2(S^2)$ the map $t \mapsto (B_t,\nu_t)$ is $L^2$ strongly continuous. By \cite{BT} joining is a $L^2$ strongly continuous operation, so $t \mapsto h_\mu(\partial \Gamma_i \vee B_t, (\nu_P)_i \vee \nu_t)$ is continuous. It cannot be constant since the join of a essentially free action and an action with a non-trivial kernel cannot be trivial, but as $t \ra 0$, $B_t$ converges to the trivial boundary. This establishes the following strengthened version of theorem \ref{middle H thm}.

\begin{corollary}
For each $\epsilon$ there is an interval $[a, b] \subseteq \calH_\text{bound}(F_d,\mu)$ where $a>\frac{d-2}{2d-2}h_{RW}(\mu) -\epsilon$
\end{corollary}

\section{Proof of Theorem \ref{hyperbolic boundaries}}

Motivated in part by classical small cancellation theory, Gromov initiated the study of hyperbolic groups. This work was refined by Ol'shanskii \cite{O}, who produced several precise results that generalized small cancellation theory to the setting of hyperbolic groups. From this we can attain a large family of distinct hyperbolic quotients of every hyperbolic group. Under mild moment assumption on the increment of a random walk, the Poisson boundary of a random walk on any hyperbolic group can be realized via the hitting measure of that walk on the Gromov boundary of the group. The action of a hyperbolic group on it's Gromov boundary is essentially free for any non-atomic Borel measure, a fact which we will leverage to prove when a hyperbolic group acts on the Poisson boundaries of distinct hyperbolic quotients, these actions are not equivariantly isomorphic.

The next proposition says that every non-elementary hyperbolic group has a countably deep, countably branching, rooted tree of non-elemtary hyperbolic quotients. Recall that the set of finite sequences of natural numbers is denoted $\bbN^{<\infty}$. 

\begin{prop}\label{manyquotients}
Every non-elementary hyperbolic group $\Gamma$ has a countable family of non-elementary hyperbolic quotients indexed by finite sequences of natural numbers. $\Gamma$ has a family of non-elementary hyperbolic quotients $\{Q_\alpha | \alpha \in\bbN^{<\infty}\}$ so that if $\alpha < \beta$ in $\bbN^{<\infty}$ (i.e $\beta$ extends the sequence $\alpha$) then $Q_\alpha > Q_\beta$ (which means $Q_\alpha \ra Q_\beta$). 
\end{prop}

\begin{proof}
First, without loss of generality, we can pass from $\Gamma$ to $\Gamma/E(\Gamma)$ and assume that $E(\Gamma) = 1$. We will proceed via induction, which will rest on the following claim.

\begin{claim}\label{claim}
If $\Gamma$ is a non-elementary hyperbolic group and $S \subset \Gamma$ is a finite subset of $\Gamma$, then
\begin{enumerate}
    \item there is a countable family of non-trivial, non-elementary hyperbolic quotients $Q_i$ for $i\in \bbN$,
    \item there is a countable family of finite ``witness" subsets $\{W_i\subset \Gamma|i\in \bbN\}$, so that
    \begin{itemize}
        \item $W_i$ injects into $\Gamma_i$,
        \item if $i\neq j$, then $W_i$ doesn't inject into $Q_j$ or $W_i$ doesn't inject into $Q_j$,
        \item for all $i$, $S$ injects into $Q_\alpha$, and is disjoint from $W_i$
    \end{itemize}
\end{enumerate}
\end{claim}

\noindent To prove this claim, We make heavy use of the following theorem of Ol'shanskii:

\begin{theorem}\label{Ol}
For $\Gamma $ a non-elementary hyperbolic group, take any finite set $A\subset \Gamma$ and suppose that $\Gamma'$ is a non-elementary subgroup of $\Gamma$ so that $E(\Gamma)$ is the largest finite subgroup of $\Gamma$ normalized by $\Gamma'$. Then $\Gamma$ has a  is a quotient $\epsilon_0:\Gamma\ra Q$ so that
\begin{itemize}
    \item $Q$ is a non-elementary hyperbolic group;
    \item $\epsilon_0 |_{\Gamma'}:\Gamma' \ra Q$ is surjective;
    \item $\epsilon_0$ is injective on $A$
\end{itemize} 
\end{theorem}

By the theorem, if we take any non-trivial, non-elementary subgroup $\Gamma' \lneq \Gamma$, and let $A= A_1 = \{1, a_1\}$, then we get a non-trivial non-elementary hyperbolic quotient $Q' = \Gamma/N_1$ so that $A_1$ injects. Now take $a_2$ to be any non identity element in $N_1$, and set $A_2 = \{1,a_1, a_2\}$. Now \ref{Ol} produces another non-trivial quotient $Q_2$ in the same way, so that $A_2$ injects, but $A_1$ does not. Proceeding by induction builds the family of quotients we desire, and proves the claim.

For the main induction, assume that we have constructed the tree of quotients to a finite depth of $n$. This means the following things exist:

\begin{enumerate}
    \item A family of quotients $Q_\alpha$ for each $\alpha in \bbN^{<n} = \cup{i=0}^{n-1} \bbN^i$, so that if $\alpha$ is an extends $\beta$, (i.e. $\alpha > \beta$) then $Q_\alpha > Q_\beta$ (i.e. $Q_\alpha\ra Q_\beta$)
    \item A collection of ``witness" subsets $\{W'_\alpha \subset \Gamma | \alpha \in \bbN^{< n} \}$, so that 
    \begin{itemize}
        \item $W'_\alpha \hookrightarrow Q_\alpha$,
        \item if $\alpha \neq \beta$ then either $W'_\alpha \not \hookrightarrow Q_\beta$ or $W'_\beta \not \hookrightarrow Q_\alpha$
    \end{itemize}
\end{enumerate} 
The purpose of the witness sets $A_\alpha$ is to ensure that the quotients $Q_\alpha$ are all distinct. 

For any natural number sequence $\alpha= (a_1,\ldots,a_{n-1})$ let $\alpha i = (a_1,\ldots,a_{n-1},i)$; we now use claim above to construct for each $i\in\bbN$ a quotient of $Q_\alpha$ denoted $Q_{\alpha i}$, and then repeat for of all of the possible indices $\alpha \in \bbN^{n-1}$. We can also ensure that for each $i$, $W_\alpha \hookrightarrow Q_{\alpha i}$ and that there is a set of witnesses $W_{\alpha i}$ so that for $i \neq j $ either $W_{\alpha i} \not \hookrightarrow Q_{\alpha j}$ or $W_{\alpha j} \not \hookrightarrow Q_{\alpha i}$. Now take as witnesses $W'_{\alpha i} = W_{\alpha i} \cup W_{\alpha }$, so the proposition is proved.

\end{proof}

\begin{proof}{\it(of theorem \ref{hyperbolic boundaries})}
First, pass to a non-elementary hyperbolic quotient of $\Gamma$ with property (T), as described in \cite{Gromov}, this will ensure that the boundaries we construct later will have entropies bounded fd Given the assumptions on $\mu$, we can identify the Poisson boundaries of $Q_\alpha$ with their Gromov boundaries \cite{Kaim}. That is $(\partial_PQ_\alpha,(\nu_P)_\alpha)\cong (\partial Q_\alpha,\nu_\alpha)$ where $\nu_\alpha$ is the unique regular, stationary measure with full support on $\partial Q_\alpha$. For $\nu_\alpha$-a.e. point $\xi \in \partial Q_\alpha$, the stabilizer of $\xi$ is $E(\Gamma)$. Considered as a $\Gamma$ space, the essential kernel of  $\Gamma \curvearrowright \partial Q_\alpha$ is precisely $E(Q_\alpha)N_\alpha$, where $N_\alpha$ is the kernel of $\Gamma \ra Q_\alpha$.

By the selection of the witness set all of these essential stabilizers are distinct. In a $\Gamma$ equivariant isomorphism of boundaries, the essential stabilizers would conjugated. As the essential stabilizers are normal, this is impossible. Therefore $\bbN^{<\infty}$ embeds into $\mathcal{BL}(\Gamma,\nu)$ 
\end{proof}

Each ``branch" of the tree of boundaries constructed has a non-trivial infimum boundary. It is interesting to ask whether the whether there is some way to distinguish them, in which case one could improve the statement to say that $\mathcal{BL}(\Gamma,\mu)$ and $\calH_\text{bound}(\Gamma,\mu)$ have the cardinality of the continuum.

\end{document}